\documentclass{article}
\usepackage{amsmath,amsfonts,amssymb,amsthm,bbm,latexsym,mathrsfs}
\usepackage{graphicx,color,epsfig,fancyhdr,dsfont}
\usepackage{enumerate}
\usepackage{hyperref}
\usepackage{indentfirst}
\usepackage{amsmath,amscd}
\usepackage{caption}
\usepackage{graphicx, subfig}
\usepackage[all]{xy}

\newtheorem{definition}{Definition}
\newtheorem{lemma}[definition]{Lemma}
\newtheorem{theorem}[definition]{Theorem}
\newtheorem{coro}[definition]{Corollary}
\newtheorem{prop}[definition]{Proposition}
\newtheorem{example}[definition]{Example}
\newtheorem{remark}[definition]{Remark}
\newtheorem{fact}[definition]{Fact}
\newtheorem{ques}[definition]{Question}

\newcommand{\Th}{\mathrm{Th}}

\newcommand{\cl}{\mathrm{cl}}

\newcommand{\pCF}{p\mathrm{CF}}

\newcommand{\Q}{\mathbb{Q}}

\newcommand{\SL}{\mathrm{SL}}

\newcommand{\M}{\mathbb{M}}

\newcommand{\Gen}{\mathrm{Gen}}

\newcommand{\R}{\mathbb{R}}

\newcommand{\Ss}{\mathbf{S}}

\newcommand{\up}{\Pi^\mathrm{def}}
\newcommand{\ups}{\Pi^\mathrm{def}_\mathrm{s}}
\newcommand{\Prob}{\mathrm{Prob}}

\newcommand{\Gg}{\mathbf{G}}

\newcommand{\Pp}{\mathbf{P}}
\newcommand{\PP}{\mathbb{P}}

\newcommand{\Ff}{\mathbf{F}}

\newcommand{\F}{\mathcal{F}}

\newcommand{\fullS}{S^*}

\begin{document}
\title{Minimal proximal definable flows over the $p$-adics}
\author{Zhentao Zhang}
\date{}
\maketitle

\begin{abstract}
Let $G$ be a definable group in an NIP theory. We prove that every minimal proximal definable $G$-flow is strongly proximal. Consequently, the universal minimal proximal definable $G$-flow $\up(G)$ coincides with the minimal strongly proximal definable $G$-flow $\up_s(G)$.

Furthermore, for a $p$-adic definable group $G$, we can compute $\up(G)$ explicitly. We show that $\up(G)$ is exactly $\up(S)$ where $S$ is the semisimple part of the definably amenable-semisimple decomposition of $G$. In addition, $\up(S)\cong S^*_\F(\Q_p)$, the space of types of full dimension on $\F$, where $\F=\Ff(\Q_p)$ for a flag variety $\Ff$ constructed from $S$.   
\end{abstract}

\section{Introduction}

Proximality is one of the basic contraction phenomena in topological dynamics. Let a group $G$ act on a compact Hausdorff space $X$. The flow $(G,X)$ is \emph{proximal} when  every pair of points can be moved arbitrarily close by the action.  The flow is \emph{strongly proximal} when the induced action
on the compact space $\Prob(X)$ of regular Borel probability measures can move every measure
arbitrarily close to a Dirac measure. Strong proximality implies proximality, but the converse fails
in general. The corresponding universal minimal objects are the universal minimal proximal flow
$\Pi(G)$ and the universal minimal strongly proximal flow $\Pi_{\mathrm{s}}(G)$. See \cite{G-Book} for details.

Amenability of $G$ is equivalent to the triviality of $\Pi_\mathrm{s}(G)$ (Theorem III.3.1 \cite{G-Book}). And strong amenability is equivalent to the triviality of $\Pi(G)$ by definition (Chapter II \cite{G-Book}).  Since an amenable group need not be
strongly amenable (for example, see \cite{Diff}). Identifying proximal and strongly proximal dynamics therefore requires additional conditions.

Definable topological dynamics studies actions whose orbit maps are definable in an ambient
first-order structure. Let $M$ be a structure in a complete theory $T$ and $G=G(M)$ a group definable in $M$. Let $X$ be a compact Hausdorff space. The flow $(G,X)$ is called a \emph{definable $G$-flow} if for every $x\in X$, the map $f_x:G\rightarrow X:g\mapsto gx$ is a \emph{definable map}, that is, for any disjoint closed subsets $C_1,C_2$ of $X$, there exists a definable subset $D$ of $G$ separating $f_x^{-1}(C_1)$ and $f_x^{-1}(C_2)$. Often it is assumed that there is a dense orbit, and sometimes a
$G$-flow $(G,X)$ with a distinguished point $x\in X$ whose orbit is dense is called a $G$-ambit. Obviously, the type space $S_G(M)$ is a definable $G$-ambit. When every type in $S_G(M)$ is definable, $S_G(M)$ is the
universal definable $G$-ambit. See \cite{GPP} for details.

Now we study $\Pi(G)$ and $\Pi_\mathrm{s}(G)$, in the above definable settings. We write $\up(G)$ and $\ups(G)$ for the universal minimal proximal definable $G$-flow and the universal minimal strongly
proximal definable $G$-flow, respectively.
Both the universal objects $\up(G)$
$\ups(G)$ exist and are unique up to isomorphism. See Corollary 1.14 \cite{KP} for $\up(G)$ and \cite{G-Book} for the classical product construction. Note that products
and factors remain definable, so the same construction gives $\up_s(G)$.

Our first result shows that NIP removes the classical gap between these universal objects.

\begin{theorem}\label{up-ups}
Assume that $T$ is NIP. Every minimal proximal definable $G$-flow is strongly proximal.
In particular, $\up(G)=\ups(G)$.
\end{theorem}

Our main applications concern groups definable in $\Q_p$, the field of $p$-adic numbers, in the language of fields. It is well-known that the theory $\pCF=\Th(\Q_p)$ is NIP.
Moreover, over the standard model $\Q_p$, every type is definable by Delon’s theorem in \cite{Delon}, and  consequently, the definable and externally definable dynamical categories agree over $\Q_p$.

The explicit study of $p$-adic definable topological dynamics was initiated in \cite{PPY}. For $G=\SL_2(\Q_p)$,
Penazzi, Pillay, and Yao proved that
$S_{\PP^1,\mathrm{na}}(\Q_p)$,  the space
of non-algebraic types over $\Q_p$ on the projective line, is a minimal proximal definable $G$-flow. And they asked:

\begin{ques}[Question 4.9 \cite{PPY}]\label{Ques-PPY}
Work in $\pCF$.
\begin{enumerate}[(i)]
\item Is $S_{\PP^1,\mathrm{na}}(\Q_p)$ the universal minimal proximal definable $\SL_2(\Q_p)$-flow?
\item Is $S_{\PP^1,\mathrm{na}}(\Q_p)$  a strongly proximal $\SL_2(\Q_p)$-flow?
\end{enumerate}
\end{ques}

Theorem \ref{up-ups} answers the second question affirmatively. It remains for us to solve the first question. We will answer the first question affirmatively.
In fact, given an arbitrary $p$-adic definable group $G$, we have a method for computing $\up(G)=\ups(G)$.

The method is based on the structural theory of $p$-adic definable groups developed in \cite{PYZ} and \cite{YZ}. The structural theory gives a \emph{definably amenable-semisimple decomposition} of $G$ when $G$ is not definably amenable, i.e.,  a definable exact sequence 
\begin{equation}\label{decomp}
1\rightarrow D\rightarrow G \stackrel{\pi}{\rightarrow} S\rightarrow 1
\end{equation}
where $D$ is definably amenable and $S$ is a finite-index subgroup of $\Ss(\Q_p)$ with $\Ss$ a connected semisimple algebraic group which is an almost product of almost $\Q_p$-simple, $\Q_p$-isotropic algebraic groups over $\Q_p$. In addition, we let $S$ be trivial and $D=G$, when $G$ is definably amenable. Note also that the decomposition is unique, in the sense that 
$D$ is unique up to commensurability. 

Then every definable $S$-flow $X$ can be treated as a definable $G$-flow via $\pi$. More specifically, we let $G$ act on $X$ as $gx=\pi(g)x$ for $g\in G$ and $x\in X$. By $X|_\pi$, we denote the induced $G$-flow. 
We will show that

\begin{theorem}\label{thm-reduce}
Assume that $T$ is NIP.
Let $G$ be a definable group. Let
$1\rightarrow D\rightarrow G \stackrel{\pi}{\rightarrow} S\rightarrow 1$ be a definable exact sequence where $D$ is definably amenable.
Then
$$\up(G)\cong \up(S)|_\pi.$$
\end{theorem}

In particular, if $G$ is a $p$-adic definable group with its definably amenable-semisimple decompostion (\ref{decomp}), $\up(G)\cong \up(S)|_\pi$. We then reduce the question to $S$.

Note that $S$ is of finite-index in $\Ss(\Q_p)$. It is not hard to show that there is no difference between studying $S$ and $\Ss(\Q_p)$ on this question. So for convenience, we assume that $S=\Ss(\Q_p)$. Let $\Pp$ be a minimal $\Q_p$-parabolic
subgroup. Then $\Ff:=\Ss/\Pp$ is a flag variety. Let $\F=\Ff(\Q_p)$. We let $\fullS_\F(\Q_p)$ be the subspace of the type space $S_\F(\Q_p)$ consisting of types of full dimension. Note that $S$ naturally acts on $\F$ and then naturally acts on $S_\F(\Q_p)$. It is easy to see that $\fullS_\F(\Q_p)$ is $S$-subflow.
We will prove that

\begin{theorem}\label{thm-S}
Working in $\pCF$. Let $G$ be a definable group with its definably amenable-semisimple decompostion (\ref{decomp}). Then, as definable $S$-flows, $$\up(S)\cong \fullS_\F(\Q_p).$$
\end{theorem}

Note that when $G=S=\SL_2(\Q_p)$, $\Ss=\SL_2$ and $\Pp<\Ss$ is the subgroup of upper triangular matrices, $\F=\PP^1$ and then
$\fullS_\F(\Q_p)=S_{\PP,\mathrm{na}}(\Q_p)$. Thus, we solved the second question of Question \ref{Ques-PPY}.

\

The paper is organized as follows. In Section \ref{Preliminaries}, we review some background on proximal dynamics. In Section \ref{NIP-up}, we prove Theorem \ref{up-ups}. In Section \ref{reduce}, we prove Theorem \ref{thm-reduce}. In
Section \ref{p-adic}, we prove Theorem \ref{thm-S}.

For notation, we always let $T$ be a complete theory and $M$ be a structure in $T$. A definable set $D$ is a set defined in $M$ with parameters from $M$. We identify a definable set with its defining formula, so $D=D(M)$. We assume basic familiarity with model theory, and our model-theoretic terminology is standard. When studying groups defined over $\Q_p$,  an algebraic group $\Gg$ is treated as an algebraic variety, which is usually studied in a very saturated algebraically closed field $\Omega$ containing $\Q_p$. Our terminology for algebraic groups is standard. Unless otherwise stated, all algebraic groups in this paper are connected.
Note that $\Gg(\Q_p)$ is a definable group in the pure field structure $\Q_p$, which is sometimes also called an $\Q_p$-algebraic group. We refer the reader to \cite{PYZ} and \cite{YZ} for properties of 
$p$-adic definable groups, in particular some special $\Q_p$-algebraic groups.

\section{Definable proximal dynamics}\label{Preliminaries}

We review some basic definitions. 

\begin{definition}
Let $G$ be a definable group, and $X$ a compact Hausdorff space. Let $G$ act on $X$. Then we call $X$ a $G$-flow. The flow $(G,X)$ is definable if for every $x\in X$, the map $f_x:G\rightarrow X:g\mapsto gx$ is  a definable map, that is, for any disjoint closed subsets $C_1,C_2$ of $X$, there exists a definable subset $D$ of $G$ separating $f_x^{-1}(C_1)$ and $f_x^{-1}(C_2)$.
\end{definition}

When $X$ is a definable $G$-flow and $x\in X$, the map $f_x:G\rightarrow X:g\mapsto gx$ can be extended to $\hat{f_x}:S_G(M)\rightarrow X:p\mapsto \lim_i f_x(g_i)$ where $g_i\in G$ (treated as realized types) converges to $p$. We write $p x$ for $\hat{f_x}(p)$ in short.  If every type
in $S_G(M)$) is definable, then $S_G(M)$  is the universal definable $G$-ambit (Proposition 3.8 \cite{GPP});
without that hypothesis it need not be a definable flow. 

Let $X$ be a compact Hausdorff space. Let $C(X)$ be the space of all continuous functions $X\rightarrow \R$. 
Let $\Prob(X)$ be the space of regular Borel probability measures. Then $\Prob(X)$ can be treated as a subspace of the dual of $C(X)$ with the weak-* topology. More specifically,  a net $(\mu_i)_i$ of $\Prob(X)$ converges to  $\mu\in \Prob(X)$ iff $\int_X f \text{ } \mathrm{d}\mu_i$ converges to $\int_X f \text{ } \mathrm{d}\mu$ for every $f\in C(X)$. Also note that $G$ naturally acts on $\Prob(X)$ and $\Prob(X)$ is a $G$-flow.

\begin{definition}
Let $G$ be a group and $X$ be a $G$-flow.
\begin{enumerate}
\item A ($G$-)subflow $Y$ of $X$ is a closed subset which is closed under the $G$-action. $X$ is minimal if there is no nonempty proper subflow.
\item Let $x,y\in X$. We say that $x,y$ are proximal if there is a net $(g_i)_i$ of $G$, such that $\lim_i g_i x=\lim_i g_i y$. $X$ is called proximal if any $x,y\in X$ are proximal.
\item $X$ is strongly proximal if for every $\mu\in \Prob(X)$, the weak-* closure of the orbit $G \mu$ contains a Dirac measure.
\end{enumerate}
\end{definition}

The universal minimal proximal definable $G$-flow $\up(G)$ and  universal minimal strongly proximal definable $G$-flow
$\ups(G)$ exist and are unique up to isomorphism. 
See Corollary 1.14 \cite{KP} for $\up(G)$. Also see \cite{G-Book} for the classical product construction. Note that products
and factors remain definable, so the same construction gives $\up_s(G)$.

We now recall two results about minimal proximality.

\begin{fact}[II Lemma 4.1 \cite{G-Book}]\label{fact-1}
Let $X$ be a minimal proximal flow. Then every endomorphism of $X$ is the identity automorphism..
\end{fact}

\begin{fact}[II Propostion 2.1 \cite{G-Book}]\label{fact-2}
Let $X$ be a minimal proximal $G$-flow. Let $x_1,\dots, x_n,z\in X$. Then for every neighborhood  $V$ of $z$, there is $g\in G$ such that $g x_j\in V$ for every $j$.
Equivalently, there is a net $(g_i)_i$ of $G$ such that $\lim_i g_i x_j=z$ for every $j$.
\end{fact}

\section{Proximality in NIP theories}\label{NIP-up}

We first give a lemma that lifts surjective morphisms between flows to the corresponding spaces of probability measures.

\begin{lemma}\label{lemma-lift}
Let $\rho:Y\rightarrow X$ be a continuous surjection between compact Hausdorff spaces. Let $\rho_*: \Prob(Y)\rightarrow \Prob(X): \mu\mapsto \rho_*(\mu)$ be the pushforward map, i.e., $\int_Y f \text{ }\mathrm{d}(\rho_*\mu)=\int_X f\circ\rho \text{ }\mathrm{d} \mu$ for every $f\in C(X)$. Then $\rho_*$ is surjective. 
\end{lemma}
\begin{proof}
The image of $\rho_*$ is compact and convex, and it contains every Dirac measure. Since  the closed convex hull of the Dirac measures $\Prob(X)$, the image of $\rho_*$ is the whole $\Prob(X)$.
\end{proof}

We now point out the key role of the NIP assumption. Note that a Keisler measure over $M$ (in the variable $x$) is a finitely additive probability measure on the Boolean algebra of definable sets (in the variable $x$). Moreover, every Keisler measure over $M$ can be uniquely extend to a regular Borel probability measure on $S_x(M)$. So there is a one-to-tone correspondence between the set of Keisler measures over $M$ and the set of regular Borel probability measures on $S_x(M)$. And types can be viewed as Dirac measures on $S_x(M)$. See \cite{Simon} for details. In NIP theories, measures can be approximated by averages of types.

\begin{fact}[Proposition 7.11 \cite{Simon}]\label{fact-NIP}
Assume that $T$ is NIP. Let $\mu$ be a Keisler measure over $M$. Let $\theta(x;y)$ be a formula and $\epsilon>0$. Then there are types $p_1,\dots ,p_m\in S_x(M)$ such that, for every $b\in M$,
$$|\mu(\theta(x;b))-\frac{|\{i:\theta(x;b)\in p_i\}|}{m}|<\epsilon.$$
\end{fact}

Now we prove Theorem \ref{up-ups}.

\begin{theorem}\label{thm-1}
Assume that $T$ is NIP. Let $X$ be a minimal proximal definable $G$-flow. Let $\mu\in \Prob(X)$ and $z\in X$. Then $\delta_z$, the Dirac measure for $z$, is in the weak-* closure of $G\mu$. 
Hence, $X$ is strongly proximal.
\end{theorem}
\begin{proof}
Choose a minimal subflow $Y\subseteq S_G(M)$ and a point $x_0\in X$.  Then map $f_{x_0}: G\rightarrow X$ induces a map $\hat{f_{x_0}}:S_G(M)\rightarrow X$. Let $\rho$ be the restriction of $\hat{f_{x_0}}$ on $Y$. Since $X$ is minimal, $\rho$ is surjective. By Lemma \ref{lemma-lift}, there exists $\nu\in \Prob(Y)$ such that $\rho_*(\nu)=\mu$.

Fix a neighborhood $U$ of $z$ and $\epsilon>0$. By regularity of compact Hausdorff spaces, choose an
open neighborhood $V$ of $z$ with the closure $\cl(V)\subseteq U$. Then there is an $M$-formula $\varphi$, such that $\rho^{-1}(\cl(V)) \subseteq Y\cap[\varphi] \subseteq \rho^{-1}(U)$ where $[\varphi]=\{p\in S_G(M): p\vdash \varphi\}$, a clopen subset of $S_G(M)$. 

Regard $\mu$ as a Keisler measure over $M$ by putting $\mu(\psi)=\mu(Y\cap[\psi])$ for every formula $\psi$. Apply Fact \ref{fact-NIP} to $\theta(x;y):= \text{``} y\in G\text{''} \wedge \varphi(yx)$. $\theta(x;y)$. There are $p_1,\dots, p_m\in S_G(M)$ such that, for every $b\in M$,
$$|\nu(\theta(x;b))-\frac{|\{i:\theta(x;b)\in p_i\}|}{m}|<\epsilon.$$
Of course, we can in addition assume that $p_1,\dots, p_m\in Y$.

By Fact \ref{fact-2}, there is $g\in G$ such that $g \rho(p_j)\in V$ for every $j=1,\dots,m$.
As $g\rho(p_j)=\rho(gp_j)$ and $\rho^{-1}(\cl(V)) \subseteq Y\cap[\varphi] \subseteq \rho^{-1}(U)$, we have that $g p_j\vdash \varphi$ for every $j$ and consequently, $\frac{|\{i:\theta(x;g)\in p_i\}|}{m}=1$. Then
$$\nu(Y\cap[\varphi])=\nu(\theta(x,g))>1-\epsilon.$$
Since $\rho^{-1}(U)\supseteq Y\cap[\varphi]$, we have 
$$g\mu(U) =\nu( \{p\in Y: \rho(gp)\in U\}\geq \nu(Y\cap[\varphi])>1-\epsilon.$$

Name the above $g$ by $g(U,\epsilon)$. Direct the pairs $(U,\epsilon)$ by shrinking $U$ to $z$ and letting $\epsilon$ tend to zero.
Then the net $(g(U,\epsilon))_{U,\epsilon}$ ensures $\lim_{U,\epsilon} g(U,\epsilon) \mu=\delta_z$.
\end{proof}

\begin{coro}
Assume that $T$ is NIP. Then $\up(G)=\ups(G)$.
\end{coro}
\begin{proof}
The classes of minimal proximal definable flows and minimal strongly proximal definable
flows coincide. Their universal objects therefore coincide.
\end{proof}

\section{Quotient by a definably amenable group}\label{reduce}

We first deal with the quotient group without additional assumptions.

\begin{lemma}\label{lemma-D}
Let $1\rightarrow D\rightarrow G \stackrel{\pi}{\rightarrow} S\rightarrow 1$ be a definable exact sequence. Let $R_D$ be the smallest  closed $G$-invariant (setwise) equivalence relation on $\up(G)$ containing $\{(x,dx):d\in D\}$. Then $\up(G)/R_D$ is an $S$-flow and 
$$\up(G)/R_D\cong \up(S).$$
Equivalently, if we treat $\up(G)/R_D$ as a $G$-flow, we have
$$\up(G)/R_D\cong \up(S)|_\pi.$$
\end{lemma}
\begin{proof}
Since $\up(G)$ is compact and $R_D$ is closed,
the quotient $\up(G)/R_D$ is a compact Hausdorff space. Clearly, $D$ acts trivially on $\up(G)/R_D$.
It is easy to check that $\up(G)/R_D$ remains minimal, proximal, and
definable as a $G$-flow. To show that it is definable as an $S$-flow, we let $C_1, C_2$ be disjoint closed subsets of $\up(G)/R_D$ and $x\in \up(G)/R_D$. Let $f_{G,x}: G\rightarrow \up(G)/R_D$ and $f_{S,x}:S\rightarrow\up(G)/R_D$ be the orbit map for $x$ in $G$ and $S$, respectively.
Then there is a definable subset $U$ of $G$ separating $f_{G,x}^{-1}(C_1)$ and $f_{G,x}^{-1}(C_2)$. Since $f_{G,x}=f_{S,x}\circ\pi$, the definable subset $V:=\{s\in S: \pi^{-1}(s)\subseteq U\}$ separates $f_{S,x}^{-1}(C_1)$ and $f_{S,x}^{-1}(C_2)$.

Now let $Y$ be any minimal proximal definable $S$-flow. It is obvious that $Y|_\pi$ is a minimal proximal definable $G$-flow. The universality of $\up(G)$ gives a surjective $G$-map $\rho:\up(G)\rightarrow Y$. Since $D$ acts trivially on $Y$, the relation $\{(x,y)\in \up(G)\times \up(G): \rho(x)=\rho(y)\}$ contains $R_D$. So $\rho$ induces a surjective $S$-map $\up(G)/R_D\rightarrow Y$. Since $Y$ is arbitrary, $\up(G)/R_D\cong \up(S)$.
\end{proof}

Then we study definable amenability.

\begin{definition}
A definable group $G$ is definably amenable if it admits a (left) $G$-invariant Keisler measure. 
\end{definition}

As we mentioned earlier, such $G$-invariant
Keisler measures correspond to $G$-invariant regular Borel probability measures on $S_G(M)$.
Then we have the following easy result.

\begin{lemma}\label{lemma-easy}
Let $G$ be definably amenable. Let $X$ be a definable $G$-flow. Then $X$ admits a $G$-invariant regular Borel probability measure.
\end{lemma}
\begin{proof}
Fix $x_0\in X$. The map $\hat{f_{x_0}}: S_G(M)\rightarrow X: p\mapsto px_0$ sends $G$-invariant regular Borel probability measures on $S_G(M)$ to regular Borel probability measures on $X$.
\end{proof}

\begin{lemma}\label{lemma-trivial-D}
Assume that $T$ is NIP. Let $D\unlhd G$ be a definably amenable normal subgroup of $G$. Let $X$ be a minimal proximal definable $G$-flow. Then $D$ acts trivially on $X$. 
\end{lemma}
\begin{proof}
By Lemma \ref{lemma-easy}, there exists a $D$-invariant measure $\mu\in \Prob(X)$.  By
Theorem \ref{up-ups}, there are a net $(g_i)_i$ in $G$ and $z\in X$ such that $\lim_i g_i\mu=\delta_z$.
The normality of $D$ implies that every $g_i\mu$ remains $D$-invariant. So the limit $\delta_z$ is $D$-invariant. So $\mathrm{Fix}(D):=\{x\in X:dx=x \text{ for every }d\in D\}$
is nonempty. It is easy to check that $\mathrm{Fix}(D)$ is closed and $G$-invariant (setwise). By the minimality of $X$, $\mathrm{Fix}(D)=X$.
\end{proof}

Then Theorem \ref{thm-reduce} is a corollary of  Lemma \ref{lemma-D} and Lemma \ref{lemma-trivial-D}.

\begin{coro}
Assume that $T$ is NIP. Let
$1\rightarrow D\rightarrow G \stackrel{\pi}{\rightarrow} S\rightarrow 1$ be a definable exact sequence where $D$ is definably amenable.
Then
$$\up(G)\cong \up(S)|_\pi.$$
\end{coro}
\begin{proof}
 Lemma \ref{lemma-trivial-D} implies that $R_D=\{(x,x): x\in \up(G)\}$. Then apply Lemma \ref{lemma-D}.
\end{proof}

\begin{remark}
Lemma \ref{lemma-trivial-D} was proved by Krupi\'{n}ski and Pillay in \cite{KP} in the case when $D=G$. Moreover, they also proved in \cite{KP} that, in NIP theories, $G$ is definably amenable iff $G$ is definably strongly amenable (i.e., $\up(G)$ is trivial by definition).
The result is also a special case of Theorem \ref{thm-reduce}. With Theorem \ref{up-ups}, in summary, we have that, in NIP theories, $G$ is definably amenable iff $\up(G)$ is trivial iff $\up_s(G)$ is trivial.
\end{remark}

We next treat the case of taking the quotient by a definably amenable group which need not be normal, and give a result needed in the next section. 

Recall that when every type in $S_G(M)$ is definable, $S_G(M)$ is the universal $G$-ambit. Moreover, for every $q\in S_G(M)$, the map $f_q:G\rightarrow S_G(M): g\mapsto g q$ induces $\hat{f_q}:S_G(M)\rightarrow S_G(M): p\mapsto pq$. The operation $*:S_G(M)\times S_G(M)\rightarrow S_G(M):(p,q)\mapsto p*q:=pq$ is associative and therefore $(S_G(M),*)$ is a semigroup. If $X$ is a definable $G$-flow, then $S_G(M)$ naturally acts on $X$ by $(p,x)\mapsto px$. And for $x\in X$, the closure of $Gx$ is $S_G(M)x$.
If we further assume NIP, then by Proposition 6.3  \cite{KP}, $\Prob(S_G(M))$ is also a definable $G$-flow. If $X$ is minimal, by the university of $S_G(M)$,
there exists a surjective $G$-map $\rho:S_G(M)\rightarrow X$. Then, by Lemma \ref{lemma-lift}, $\rho$ induces a surjective $G$-map $\rho_*:\Prob(S_G(M))\rightarrow \Prob(X)$. Since $\Prob(S_G(M))$ is a definable $G$-flow, so is $\Prob(X)$.

\begin{prop}\label{prop-use}
Assume that $T$ is NIP. Assume that every type in $S_G(M)$ is definable. Let $D\leq G$ be a definably amenable subgroup of $G$. Assume that $Z:=G/D$ is definable (not merely interpretable) and there exists a definable section $\alpha:Z\rightarrow G$ for the projection $\pi: G\rightarrow Z$. Suppose that $u\in S_G(M)$ and $\xi\in S_Z(M)$ such that 
\begin{enumerate}
\item $uz=\xi$ for every $z\in S_Z(M)$.
\item $X:=S_G(M)\xi$ is minimal.
\end{enumerate}
Then $X$ is the universal minimal proximal definable $G$-flow.
\end{prop}

\begin{proof}
It is easy to check that $X$ is a minimal proximal definable $G$-flow. 

Let $Y$ be an arbitrary minimal proximal definable $G$-flow. As we mentioned above, $\Prob(Y)$ is also a definable $G$-flow.

The definable amenability of $D$ provides a $D$-invariant Keisler measure $\nu$ on $D$. We treat $\mu$ as a regular Borel probability measure on $S_D(M)$. Pick an arbitrary element in $\lambda\in \Prob(Y)$. Then the orbit map $f_\lambda:D\rightarrow \Prob(Y)$ can be extend to $\hat{f_\lambda}: S_D(M)\rightarrow \Prob(Y)$. 
Let $\mu=\hat{f_\lambda}(\nu)\in \Prob(Y)$.
Then $\mu$ is $D$-invariant. 

Let $f: Z=G/D\rightarrow  \Prob(Y): z=gD\mapsto g\mu$. It is well-defined and is a $G$-map. Also note that $f(z)=\alpha(z) \mu$.  Since $Z$ and $\alpha$ are all definable. The map $f$ can be naturally extended to a $G$-map $\hat{f}: S_Z(M)\rightarrow \Prob(Y)$. Moreover, when treating $S_Z(M)$ and $\Prob(Y)$ as $S_G(M)$-actions, we also have $\hat{f}(pz)=p\hat{f}(z)$ for any $p\in S_G(M)$ and $z\in S_Z(M)$.

By Theorem \ref{up-ups}, $Y$ is strongly proximal. Hence there are net $(g_i)_i$ in $G$ and $y_0\in Y$, $\lim_i g_i \mu=\delta_{y_0}$. After taking a subnet, we let $(g_i)_i$  converge to some $z_0\in S_Z(M)$. 
Then $$\hat{f}(\xi)=\hat{f}(uz_0)=u\hat{f}(z_0)=u\delta_{y_0}=\delta_{u y_0}.$$
Moreover, every $x\in X$ has a form $x=p \xi$ for some $p\in S_G(M)$, and then 
$$\hat{f}(x)=\hat{f}(p\xi)=\delta_{puy_0}.$$
Let $\delta:Y\rightarrow\Prob(Y):y\mapsto\delta_y$ be the  Dirac embedding. We have $\hat{f}(X)\subseteq \delta(Y)$. The map $\delta^{-1}\circ \hat{f}|_X: X\rightarrow Y$ is a continuous $G$-map. Its image is a subflow of the minimal flow $Y$, so it is surjective.  Since $Y$ is picked arbitrarily, $X$ has the universal property. 
\end{proof}

\section{On $p$-adic definable groups}\label{p-adic}

From now on, we work in $T=\pCF$, theory of $\Q_p$ in the language of fields. More precisely, in this section, we work in the structure $M=\Q_p$. Note that every type over $\Q_p$ is definable (see \cite{Delon}) and it is well-known that $\pCF$ is NIP. Moreover, $\pCF$ admits definable Skolem functions, i.e., every definable function admits a definable section (using the same parameters) (see \cite{Skolem}). Thus, the assumptions on 
$T$ and $M$  Proposition~\ref{prop-use} are satisfied, and the definable section $\alpha$ is automatically available. 

For the properties of definable groups and definable topological dynamics in $\pCF$, we refer the reader to \cite{PYZ} and \cite{YZ}.

\begin{fact}[\cite{YZ}]
Let $G$ be a $p$-adic definable group. Then a definable exact sequence 
\begin{equation}\label{decomp-1}
1\rightarrow D\rightarrow G \stackrel{\pi}{\rightarrow} S\rightarrow 1
\end{equation}
where $D$ is definably amenable and $S$ is either trivial (when $G$ is definably amenable) or a finite-index subgroup of $\Ss(\Q_p)$ with $\Ss$ a connected semisimple algebraic group which is an almost product of almost $\Q_p$-simple, $\Q_p$-isotropic algebraic groups over $\Q_p$ (when $G$ is not definably amenable).
\end{fact}

We now fix the notation from (\ref{decomp-1}).
By Theorem \ref{thm-reduce}, we can reduce the computation of $\up(G)$ to that of $\up(S)$. So we assume that $S$ is not trivial. By the proof of Theorem 4.4 \cite{G-Book} to the definable context, $\up(S)$ can be treated as an $\Ss(\Q_p)$-flow and $\up(S)\cong\up(\Ss(\Q_p))$ as $\Ss(\Q_p)$-flows. Thus, for convenience, we assume that $S=\Ss(\Q_p)$.

Let $\Pp$ be a minimal $\Q_p$-parabolic subgroup of $\Ss$. Let $\Ff:=\Ss/\Pp$. Let $P:=\Pp(\Q_p)$, $\F:=S/P$ and $o=P\in \F$. Then  $F=\Ff(\Q_p)$, and in fact, there is a definable compact subgroup $K$ of $S$, such that $S=KP$ and $\F=K/(K\cap P)$ (Iwasawa decomposition, see Section 5.3 \cite{Tits}). The definability of $K$ can be guaranteed by \cite{PY}. Also note that $P$ is definably amenable, and more precisely, it is a definably amenable
component of $S$ (see \cite{PYZ} and\cite{YZ}).

Recall that a definable set $X$ of $K$ is generic if $K$ can be covered by finitely many translates of $X$. A type in $S_K(\Q_p)$ is generic if all definable sets contained in it are generic. Let $\Gen(K)$ be the subspace of $S_K(\Q_p)$ consisting of all generic types.
It is well-known that, as $K$ is compact, $\Gen(K)$ is nonempty and is the unique minimal subflow of $S_K(\Q_p)$ (see \cite{YZ} for details). 

Now we generalize the genericity to actions. 
For a definable action of $K$ on a definable set $Z$ (i.e., the action map $K\times Z\rightarrow Z$ is definable), we say that a definable set $X$ of $Z$ is $K$-generic, if $Z$ can be covered by finitely many $K$-translates of $X$. Accordingly, a type in $S_Z(\Q_p)$ is defined to be $K$-generic if all definable sets contained in it are $K$-generic. If $Z$ is a definable manifold, we let $S^*_Z(\Q_p)$ be the subspace of $S_Z(\Q_p)$ consisting of types with the full dimension. Clearly, $S^*_Z(\Q_p)$ is a $K$-subflow of $S_K(\Q_p)$.

\begin{lemma}\label{easy-full-dim}
Let $K$ act transitively and definably on
a compact definable manifold $Z$. Then a definable subset $X\subseteq Z$ is $K$-generic iff $X$ is of full dimension in $Z$. Equivalently, $S_Z^*(\Q_p)$ is exactly the subspace of all $K$-generic types.
\end{lemma}
\begin{proof}
Note that, $\dim(X_1\cup X_2)=\max(\dim(X_1), \dim(X_2))$ for any definable subsets $X_1,X_2\subseteq Z$.  Thus, $X$ is of full dimension when it is $K$-generic. Conversely, a full-dimensional definable subset contains a nonempty open set $U$. The translates of $U$ cover the compact space $Z$ by transitivity of the action, and compactness gives a finite
subcover.
\end{proof}

Let $\rho: K\rightarrow \F=K/(K\cap P):x\mapsto xo$ be the natural projection. Then every fiber of $\rho$ has dimension $\dim(K\cap P)$.
By computing dimensions on the fibers, we see that the induced map $\rho:S_K(\Q_p)\rightarrow S_\F(\Q_p)$ sends $\Gen(K)$ to $S^*_\F(\Q_p)$. And it is easy to check that $\rho(\Gen(K))=S^*_\F(\Q_p)$. Then we have

\begin{lemma}\label{lemma-minimal}
$S^*_\F(\Q_p)$ is the unique minimal $K$-subflow, and also the unique minimal $S$-subflow, of $S_\F(\Q_p)$.
\end{lemma}
\begin{proof}
The image of the unique minimal $K$-subflow is minimal. Conversely, the inverse image of any
minimal $K$-subflow of $S_\F(\Q_p)$ contains $\Gen(K)$, proving uniqueness. Finally, $S^*_\F(\Q_p)$ is $S$-invariant (setwise)
because dimension is preserved by the action. Every nonempty closed $S$-subflow contains a minimal $K$-subflow, hence contains $S^*_\F(\Q_p)$, which proves the assertion for $S$.
\end{proof}

We now seek the special types required in Proposition \ref{prop-use}.

Let $\M$ be a very saturated elementary extension of $\Q_p$. Let $\mu_S$ be the type-definable infinitesimal
subgroup, the intersection of all $\Q_p$-definable neighborhoods (in $\M$) of the identity of $S(\M)$. Since we work in $\Q_p$ where global objects cannot be handled directly, we treat $\mu_S$ as the partial type consisting of all $\Q_p$-definable neighborhoods of the identity of $S$. 
The type-definable component $P^{00}$ of $P$ is also given in $\M$, to be the smallest $\Q_p$-type-definable subgroups of $P(\M)$ of small index in $P(\M)$. We also treat $P^{00}$ as a partial type over $\Q_p$.

By Proposition 3.5.4 \cite{YZ}, we pick $r\in \Gen(K)$ such that $r\vdash \mu_S$ and $r*s=r$ for every $s\vdash \mu_S$. Here, $*$ is the product on $S_S(\Q_p)$. 

Recall that, from \cite{YZ}, $P$ is definably amenable and, more precisely,
a definably amenable component of $S$. 
By Fact 2.3.2 \cite{YZ}, we pick $q\in S_P(M)$ whose  corresponding external/global type is $P^{00}$-invariant. 

\begin{fact}[Corollary 3.3.6 \cite{YZ}]\label{fact-invariant}
Let $p\in S_S(\Q_p)$ be $P^{00}(\Q_p)$-invariant. Let $a$ realize $p$ (in $\M$). Then there is $\epsilon\in \mu_S(\M)$ and $b\in P(\M)$ such that $a=\epsilon b$.
\end{fact}

Let $u:=r*q\in S_S(\Q_p)$ and $\xi:=r o\in S^*_\F(\Q_p)$. Since $q$ concentrates on $P$, it fixes $o$, consequently, $uo=\xi$.

\begin{lemma}\label{lemma-compute}
$uz =\xi$ for every $z\in S_\F(\Q_p)$.
\end{lemma}
\begin{proof}
Lift $z\in S_\F(\Q_p)$ to $t\in S_S(\Q_p)$ by definable Skolem functions, so $z = to$. Since $q$ is $P^{00}$-invariant, so is $q*t$. Fact \ref{fact-invariant} implies that $q*t$ concentrates on $\mu_S P$.
Thus, there exists $s\vdash \mu_S$ such that $(q *t)o = so$. Then we have
$$uz =(r*q *t)o=(r * s)o = ro= \xi.$$
\end{proof}

Now we can prove Theorem \ref{thm-S}.

\begin{coro}
$\up(S)\cong \fullS_\F(\Q_p)$.
\end{coro}
\begin{proof}
We apply Proposition \ref{prop-use} to $P\leq S$. Lemma \ref{lemma-compute} says that $uz=\xi$ for every $z\in S_\F(\Q_p)$. The minimality of $S^*_\F(\Q_p)=S_S(\Q_p)\xi$ follows from Lemma \ref{lemma-minimal}. Hence, all hypotheses of Proposition~\ref{prop-use} are satisfied.
\end{proof}

\begin{example}
Let $\Ss=\SL_2$ and $\Pp$ be the upper triangular Borel subgroup. Then $\F=\mathbb{P}^1$. So $S_{\PP^1,\mathrm{na}}(\Q_p)=S^*_{\mathbb{P}^1}(\Q_p)$ is the universal minimal proximal definable $\SL_2(\Q_p)$-flow.
\end{example}

\section*{Declaration of AI use}

The author used ChatGPT (OpenAI) for assistance with language editing and for discussing and checking mathematical arguments during the preparation of this manuscript. All mathematical statements, arguments, and proofs included in the final version were independently verified by the author, who takes full responsibility for the content of the article.

\end{document}